\documentclass[reqno]{amsart}

\usepackage{amssymb,amsfonts,amsmath,amsthm}
\usepackage{mathrsfs}
\usepackage{graphicx}
\usepackage{color}
\usepackage{verbatim}

\newtheorem{thm}{Theorem}[section]

\newtheorem{lem}[thm]{Lemma}
\newtheorem{cor}[thm]{Corollary}

\theoremstyle{definition}

\newtheorem*{ackn}{Acknowledgements}

\theoremstyle{remark}
\newtheorem{rmk}[thm]{Remark}
\newtheorem*{cla}{Claim}

\numberwithin{equation}{section}

\def\L{\mathscr{L}}

\def\AA{\mathcal{A}}
\def\BB{\mathcal{B}}
\def\V{\mathcal{V}}
\def\XX{\mathcal{X}}
\def\YY{\mathcal{Y}}
\def\eps{\varepsilon}
\def\es{\varnothing}

\def\ol#1{\overline{#1}}

\def\nn{\mathbf{n}}
\def\tb{\mathsf{t}}
\def\uu{\mathsf{u}}
\def\vv{\mathsf{v}}
\def\ww{\mathsf{w}}

\def\ut{\widetilde{\mathsf{u}}}
\def\uh{\widehat{\mathsf{u}}}
\def\id{\approx}

\begin{document}


\title[Finite bands are finitely related]%
{Finite bands are finitely related} 


\author{IGOR DOLINKA}

\address{Department of Mathematics and Informatics, University of Novi Sad, Trg Dositeja Obradovi\'ca 4, 
21101 Novi Sad, Serbia}

\email{dockie@dmi.uns.ac.rs}

\thanks{This research was supported by the Ministry of Education, Science, and Technological Development
of the Republic of Serbia through the grant No.174019.}








\begin{abstract}
We prove that every finite idempotent semigroup (band) is finitely related, which means that the clone of its
term operations (i.e.\ operations induced by words) is determined by finitely many relations. This solves an 
open problem posed by Peter Mayr in 2013.
\end{abstract}


\maketitle


\section{Introduction}
\label{sec:intro}

Let $\mathbf{A}=(A,\mathcal{F})$ be an algebra, where $\mathcal{F}$ is a family of finitary operations on the 
set $A$. As is well known, any term $\tb(x_1,\dots,x_n)$ (of the same similarity type as $\mathbf{A}$) induces, by
interpretation in $\mathbf{A}$, an operation $\tb^\mathbf{A}:A^n\to A$. Operations obtained in this way are the
\emph{term operations} of $\mathbf{A}$, and the collection of all term operations of this algebra is denoted by 
$\mathrm{Clo}(\mathbf{A})$ and called the \emph{clone} of $\mathbf{A}$.

For an $n$-ary operation $f$ on a set $X$ ($f:X^n\to X$) and a $k$-ary relation $\rho$ on the same set ($\rho
\subseteq X^k$) we say that $f$ \emph{preserves} $\rho$ if for any $\ol{a}_1,\dots,\ol{a}_k\in X^n$ (where
$\ol{a}_i=(a_{i,1},\dots,a_{i,n})$ for $1\leq i\leq k$) such that  
$ (a_{1,j},\dots,a_{k,j})\in\rho $
for all $1\leq j\leq n$ we have $(f(\ol{a}_1),\dots,f(\ol{a}_k))\in\rho$; in other words, 
either $\rho=\es$ or $\rho$ is a subalgebra of the $k$th direct power of the algebra $(X,f)$. For 
a family $\mathcal{R}$ of (finitary) relations on $X$ we denote by $\mathrm{Pol}(\mathcal{R})$ the set of all
operations on $X$ preserving all relations from $\mathcal{R}$ -- these are the \emph{polymorphisms} of the
relational structure $(X,\mathcal{R})$. Conversely, if $\mathcal{O}$ is a family of operations on $X$, the set 
of all relations preserved by all operations from $\mathcal{O}$ is denoted by $\mathrm{Inv}(\mathcal{O})$. A
foundational result in clone theory tells us that, for finite $\mathbf{A}$, $\mathrm{Clo}(\mathbf{A})=
\mathrm{Pol}(\mathrm{Inv}(\mathcal{F}))$, and in this sense the clone of $\mathbf{A}$ is determined by the 
collection of relations $\mathrm{Inv}(\mathcal{F})$: an operation on $A$ arises from a term if and only if
it preserves all the relations that are preserved by $\mathcal{F}$.

The set of relations $\mathrm{Inv}(\mathcal{F})$ is always infinite. However, it may happen that there is in fact
a \emph{finite} subset $\mathcal{R}\subseteq\mathrm{Inv}(\mathcal{F})$ such that $\mathrm{Clo}(\mathbf{A})=
\mathrm{Pol}(\mathcal{R})$, so that the clone of $\mathbf{A}$ is determined by a finite set of relations (which
then can be reduced to a single relation). In such a case we say that the algebra $\mathbf{A}$ is \emph{finitely 
related}. 

Finitely related algebras recently received a lot of attention in universal algebra \cite{BS} and its applications 
in computer science because of their link, exhibited e.g.\ by \cite{AMM,Few}, to other key algebraic properties
pertinent to the algebraic approach in studying the computation complexity of constraint satisfaction problems
(CSP) \cite{BJK}. See also \cite{Ai,Ba,Ba-JEMS,MMM} for some recent developments and important results regarding finitely 
related general algebras. For example, a major result of Aichinger et al.\ \cite{AMM} implies that any finite 
algebra having a Mal'cev term operation (including all finite groups and rings) is finitely related, a 
consequence of which is that there are only countably many Mal'cev clones on any finite set.

Concerning the study of finitely related (finite) semigroups, the seminal paper is the one of Davey et al.\
\cite{DJPSz}, where it was shown (among other things) that finite semigroups that are either commutative or 
nilpotent enjoy the finitely related property. This paper was followed by an extensive study of Mayr \cite{Ma},
who exhibited the first example of a non-finitely related finite semigroup (not too surprisingly, this was
the `infamous' six-element Brandt monoid $B_2^1$, which seems to behave badly with respect to almost any 
conceivable equational property of semigroup varieties). Also, Mayr proved that any finite \emph{regular} band
(an idempotent semigroup satisfying the identity $xyxzx\id xyzx$) is finitely related. 
The question of whether \emph{all} finite bands are finitely related is left as an open problem (Problem 6.3); 
the same question is mentioned following Problem 7.2 in \cite{GAIA}. It is exactly this problem that we aim to 
address in the present paper; namely, we prove the following main result.

\begin{thm}\label{main}
Let $S$ be a finite idempotent semigroup. Then $S$ is finitely related.
\end{thm}

Here is the brief outline of the paper. The next, preliminary section is divided into three parts. First, we are
going to invoke few criteria (from \cite{DJPSz,Ma}) for a finite algebra to be finitely related. Along the way,
we are going to introduce a handy concept of an $n$-scheme of terms, and specialise all these concepts to 
semigroups and words, respectively. Then we are going to review the lattice of all varieties of bands (idempotent 
semigroups) and proceed to a full effective description of their equational theories. Finally, we complete
the second section by some auxiliary results that will be used in the proof of the previous theorem, presented
in the third section. This proof uses induction on the `height' of the variety that $S$ generates in the lattice
of all band varieties, depicted in Fig.\ \ref{fig:bandv}, taking the mentioned result on regular bands as the
induction basis. We first resolve the case when $S$ generates one of the irreducible varieties in that 
lattice (Theorem \ref{thm:Tn}) and then demonstrate how to derive Theorem \ref{main} in its full generality from 
the irreducible case (Theorem \ref{thm:prod}).


\section{Preliminaries}
\label{sec:prelim}

\subsection{Term schemes and criteria for finitely related finite algebras}
\label{subsec:criteria}

Let $f:A^n\to A$ be an $n$-ary operation on the set $A$. For $i,j\in\nn=\{1,\dots,n\}$, $i<j$, we define 
an operation $f_{ij}:A^n\to A$ (sometimes called an \emph{identification minor} of $f$) by
$$
f_{ij}(x_1,\dots,x_n) = f(x_1,\dots,x_{i-1},x_j,x_{i+1},\dots,x_j,\dots,x_n);
$$
in other words, $f_{ij}$ is obtained from $f$ by identifying the variable $x_i$ with $x_j$. Similarly, if
$\tb=\tb(x_1,\dots,x_n)$ is a term of a given similarity type, then by $\tb^{(ij)}$ we denote the term
obtained from $\tb$ by replacing each occurrence of $x_i$ in it by $x_j$; furthermore, by convention we
define $\tb^{(ji)}$ to be $\tb^{(ij)}$.

As usual, we say that an operation $f:A^n\to A$ \emph{depends} on its $i$th variable $x_i$, $i\in\nn$, if there
exist $a_1,\dots,a_{i-1},a_{i+1},\dots,a_n,b,c\in A$ such that
$$
f(a_1,\dots,a_{i-1},b,a_{i+1},\dots,a_n) \neq f(a_1,\dots,a_{i-1},c,a_{i+1},\dots,a_n).
$$
For example, the identification minor $f_{ij}$ does not depend on $x_i$. In turn, as another example, the $i$th 
projection operation fails to depend on any of its variables but $x_i$.

Now let $\V$ be a variety and assume 
$$\mathcal{S}=\{\tb_{ij}:\ 1\leq i<j\leq n\}$$
is a family of terms over $X_n=\{x_1,\dots,x_n\}$ of the similarity type of $\V$ satisfying the following 
conditions:
\begin{itemize}
\item[(D)] For each $\mathbf{A}\in\V$, the term operation $\tb_{ij}^{\mathbf{A}}(x_1,\dots,x_n)$ does not 
depend on the variable $x_i$.
\item[(C1)] For any four distinct $1\leq i,j,p,q\leq m$ such that $i<j$ and $p<q$, $\V$ satisfies the identity
$$ \tb_{ij}^{(pq)} \id \tb_{pq}^{(ij)}. $$
\item[(C2)] For any three $1\leq i<j<k\leq n$, $\V$ satisfies the identities
$$ \tb_{ij}^{(jk)} \id \tb_{jk}^{(ik)} \id \tb_{ik}^{(jk)}. $$
\end{itemize} 
The condition (D) is called \emph{dependency}, while (C1) and (C2) are the \emph{consistency} conditions; the
family $\mathcal{S}$ is called an \emph{$n$-scheme of terms} for $\V$. Bearing in mind Definition 2.3 and 
Notation 2.7 from \cite{DJPSz} it is an easy exercise to see that this is just equivalent to the notion of 
an $(n,n-1)$-scheme introduced in that paper.

Similarly, it is very easy to see that the following holds.

\begin{lem}\label{lem:scheme}
Let $\V$ be an arbitrary variety and $\tb(x_1,\dots,x_n)$ a term of the same similarity type as $\V$. The
family of all identification minors of $\tb$,
$$
\{\tb^{(ij)}:\ 1\leq i<j\leq n\},
$$
is an $n$-scheme of terms for $\V$.
\end{lem}

Again following \cite{DJPSz}, we say that an $n$-scheme $\mathcal{S}=\{\tb_{ij}:\ 1\leq i<j\leq m\}$ for $\V$ 
\emph{comes from the term} $\tb$ if $\mathcal{S}$ is $\V$-equivalent to the family of all identification minors
of $\tb$ in the sense that $\V$ satisfies the identities
$$ \tb_{ij} \id \tb^{(ij)} $$
for all $1\leq i<j\leq n$.

The following characterisation is part of Theorem 2.9 from \cite{DJPSz}, see also \cite{Ja,Ma,Ro,RS}.

\begin{thm}\label{thm:finrel}
Let $\mathbf{A}$ be a finite algebra, and let $\V$ be the variety generated by $\mathbf{A}$. 
The following are equivalent: 
\begin{itemize}
\item[(1)] $\mathbf{A}$ is finitely related.
\item[(2)] There exists $n_0\geq|A|$ such that for all $n>n_0$, every $n$-scheme of terms for $\V$ comes from 
a term.
\item[(3)] There exists $n_0\geq|A|$ such that for all $n>n_0$, an operation $f:A^n\to A$ is a term operation
of $\mathbf{A}$ whenever all minors $f_{ij}$ ($i,j\in\nn$, $i<j$) are term operations of $\mathbf{A}$.
\end{itemize}
\end{thm}

Here we immediately invoke Remark 2.12 from \cite{DJPSz} which provides an argument that in checking if a
finite algebra $\mathbf{A}$ satisfies condition (3) above, it can be assumed that the operation $f:A^n\to A$
is \emph{essential}, i.e.\ that it depends on all of its variables. Therefore, we will effectively use the previous
theorem in the form where the condition (3) is replaced by
\begin{itemize}
\item[(3')] There exists $n_0\geq|A|$ such that for all $n>n_0$, an operation $f:A^n\to A$ depending on all of
its variables is a term operation of $\mathbf{A}$ whenever all minors $f_{ij}$ ($i,j\in\nn$, $i<j$) are term 
operations of $\mathbf{A}$.
\end{itemize} 

Notice that for semigroups and semigroup varieties all definitions and results mentioned above remain valid when 
we replace terms by words (i.e.\ elements of the free semigroup $X_n^+$), term operations by operations induced
by words, etc., even though, strictly speaking, words are not terms. However, it is true that every term operation 
on a semigroup is induced by a word and vice versa. Therefore, it is meaningful to define the notion of an
$n$-scheme of words for a semigroup variety, and the `semigroup version' of Theorem \ref{thm:finrel} holds
as well: a finite semigroup $S$ is finitely related if and only if every $n$-scheme of words for the variety 
generated by $S$ comes from a single word (provided $n$ is large enough) if and only if the condition (3') holds
with respect to operations induced by words. It was proved in \cite{Ma} that this is the case when $S$ is a 
regular band, that is, a finite idempotent semigroup satisfying $xyxzx\id xyzx$. The aim of this paper is to 
extend this to \emph{all} finite bands in an inductive manner, taking the result of \cite{Ma} as a basis of 
the induction. Hence, in the next subsection we need to gather some basic information about varieties of bands 
and few auxiliary results describing their equational theories.

\subsection{Varieties of bands and their equational theories}
\label{subsec:bandv}

The variety of all bands will be denoted by $\BB$. It is known that any semigroup variety containing $\BB$ 
(including $\BB$ itself) is not generated by a finite semigroup \cite{OSa}, while every proper subvariety of 
$\BB$ is finitely generated (see \cite{Bi,Fe,Ge}). 

In what follows, we adopt the notation of \cite{GP} for certain operators acting on words; we briefly recall it
for completeness. For a word $\ww$ over a (finite) alphabet $X$, let $c(\ww)$ denote the \emph{content} of $\ww$, 
the set of all letters occurring in $\ww$. Further, let $s(\ww)$ denote the longest prefix of $\ww$ containing all 
but one of the letters from $c(\ww)$ (so that $|c(s(\ww))|=|c(\ww)|-1$ and $s(\ww)$ is maximal with this property), 
while $\sigma(\ww)$ is the last letter to occur in $\ww$ from the left (implying that $s(\ww)\sigma(\ww)$ is the 
shortest prefix of $\ww$ with the same content as $\ww$). Dually, let $e(\ww)$ be the longest suffix of $\ww$
containing all but one of the letters from $c(\ww)$, and let $\eps(\ww)$ be the last letter in $\ww$ to occur 
from the right.  

Define an operator $b$ on words by induction on the size of their content such that $b(\es)=\es$ (here $\es$ 
stands for the empty word) and 
$$
b(\ww) = bs(\ww)\sigma(\ww)\eps(\ww)be(\ww)
$$
(note the concatenation of word operators, where $bs(\ww)$ means $b(s(\ww))$, etc.).

\begin{thm}[cf.\ Lemma 2.7 and Theorem 2.9 of \cite{GP1}] \label{thm:bwp}
Let $X$ be an alphabet and $\uu,\vv\in X^+$. Then $\uu\id\vv$ holds in $\BB$ if and only if $b(\uu)=b(\vv)$.
In particular, for any word $\ww$ we have that the identity
$$
\ww \id s(\ww)\sigma(\ww)\eps(\ww)e(\ww)
$$
holds in any band.
\end{thm}

This immediately implies the following well-known fact (recorded, e.g., as Corollary 2.3 in \cite{GP}).

\begin{cor}\label{cor:bwp}
Let $\uu,\vv,\ww$ be words such that $c(\vv)\subseteq c(\uu)=c(\ww)$. Then the identity $\uu\vv\ww\id\uu\ww$
holds in $\BB$.
\end{cor}

\begin{proof}
Under the given conditions, $s(\uu\vv\ww)=s(\uu)=s(\uu\ww)$ and $\sigma(\uu\vv\ww)=\sigma(\uu)=\sigma(\uu\ww)$,
and, dually, $e(\uu\vv\ww)=e(\ww)=e(\uu\ww)$ and $\eps(\uu\vv\ww)=\eps(\ww)=\eps(\uu\ww)$, so the result
follows by Theorem \ref{thm:bwp}.
\end{proof}

The description of the lattice $\L(\BB)$ of all subvarieties of $\BB$ is today still considered as one of the
most glaring success stories of the theory of semigroup varieties; it was achieved independently by Biryukov
\cite{Bi}, Fennemore \cite{Fe}, and Gerhard \cite{Ge}, whereas a more `economical' and coherent treatment of
the subject was given later in \cite{GP}. Figure \ref{fig:bandv} provides a diagram of this lattice: here 
$\mathcal{SL}$, $\mathcal{LZ}$ and $\mathcal{RZ}$ are respectively the varieties of semilattices, left zero 
and right zero bands.

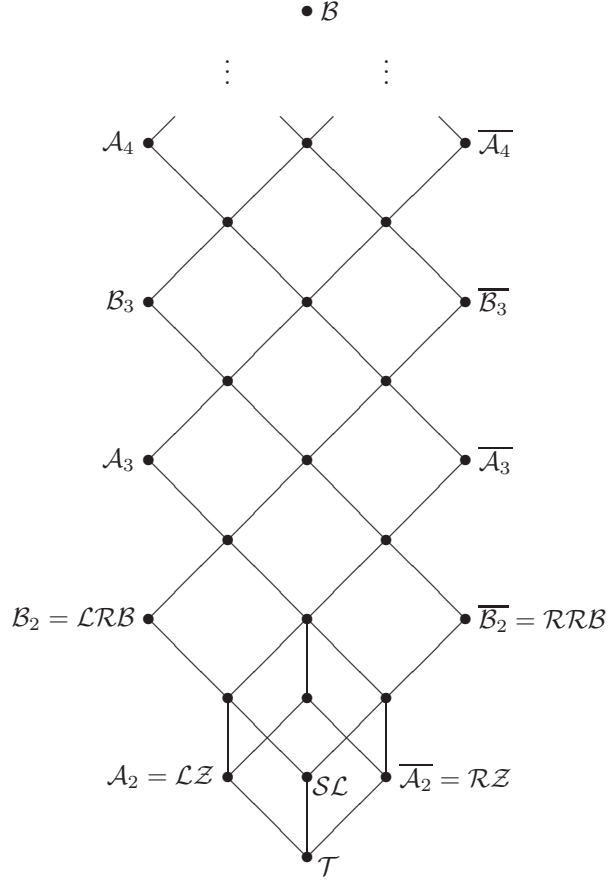
\begin{figure}[t]\centering
\begin{picture}(140.00,335.00)
\put(70.00,325.00){\circle*{4}}
\put(70.00,5.00){\circle*{4}} \put(40.00,35.00){\circle*{4}} \put(70.00,35.00){\circle*{4}}
\put(100.00,35.00){\circle*{4}} \put(40.00,65.00){\circle*{4}} \put(70.00,65.00){\circle*{4}}
\put(100.00,65.00){\circle*{4}} \put(10.00,95.00){\circle*{4}} \put(70.00,95.00){\circle*{4}}
\put(130.00,95.00){\circle*{4}} \put(40.00,125.00){\circle*{4}} \put(100.00,125.00){\circle*{4}}
\put(10.00,155.00){\circle*{4}} \put(70.00,155.00){\circle*{4}} \put(130.00,155.00){\circle*{4}}
\put(40.00,185.00){\circle*{4}} \put(100.00,185.00){\circle*{4}} \put(10.00,215.00){\circle*{4}}
\put(70.00,215.00){\circle*{4}} \put(130.00,215.00){\circle*{4}} \put(40.00,245.00){\circle*{4}}
\put(100.00,245.00){\circle*{4}} \put(10.00,275.00){\circle*{4}} \put(70.00,275.00){\circle*{4}}
\put(130.00,275.00){\circle*{4}} \put(70.00,5.00){\line(1,1){30.00}} \put(70.00,5.00){\line(-1,1){30.00}}
\put(40.00,35.00){\line(1,1){30.00}} \put(100.00,35.00){\line(-1,1){30.00}} \put(40.00,65.00){\line(1,1){30.00}}
\put(100.00,65.00){\line(-1,1){30.00}} \put(40.00,65.00){\line(0,-1){30.00}} \put(100.00,65.00){\line(0,-1){30.00}}
\put(70.00,35.00){\line(0,-1){30.00}} \put(70.00,95.00){\line(0,-1){30.00}} \put(70.00,35.00){\line(1,1){60.00}}
\put(70.00,35.00){\line(-1,1){60.00}} \put(10.00,95.00){\line(1,1){60.00}} \put(130.00,95.00){\line(-1,1){60.00}}
\put(70.00,95.00){\line(1,1){60.00}} \put(70.00,95.00){\line(-1,1){60.00}} \put(10.00,155.00){\line(1,1){60.00}}
\put(130.00,155.00){\line(-1,1){60.00}} \put(70.00,155.00){\line(1,1){60.00}} \put(70.00,155.00){\line(-1,1){60.00}}
\put(10.00,215.00){\line(1,1){70.00}} \put(130.00,215.00){\line(-1,1){70.00}} \put(70.00,215.00){\line(1,1){60.00}}
\put(70.00,215.00){\line(-1,1){60.00}} \put(10.00,275.00){\line(1,1){10.00}} \put(130.00,275.00){\line(-1,1){10.00}}
\put(40.00,305.00){\makebox(0,0)[cc]{$\vdots$}} 
\put(100.00,305.00){\makebox(0,0)[cc]{$\vdots$}}
\put(74.00,5.00){\makebox(0,0)[lt]{$\mathcal{T}$}} 
\put(72.00,35.00){\makebox(0,0)[lt]{$\mathcal{SL}$}}
\put(35.00,35.00){\makebox(0,0)[rc]{$\AA_2=\mathcal{LZ}$}} 
\put(105.00,35.00){\makebox(0,0)[lc]{$\ol{\AA_2}=\mathcal{RZ}$}} 
\put(5.00,95.00){\makebox(0,0)[rc]{$\BB_2=\mathcal{LRB}$}}
\put(135.00,95.00){\makebox(0,0)[lc]{$\ol{\BB_2}=\mathcal{RRB}$}} 
\put(5.00,155.00){\makebox(0,0)[rc]{$\AA_3$}}
\put(135.00,155.00){\makebox(0,0)[lc]{$\ol{\AA_3}$}} 
\put(5.00,215.00){\makebox(0,0)[rc]{$\BB_3$}}
\put(135.00,215.00){\makebox(0,0)[lc]{$\ol{\BB_3}$}} 
\put(5.00,275.00){\makebox(0,0)[rc]{$\AA_4$}}
\put(135.00,275.00){\makebox(0,0)[lc]{$\ol{\AA_4}$}}
\put(75.00,325.00){\makebox(0,0)[lc]{$\BB$}}
\end{picture}
\caption{The lattice of all varieties of bands}\label{fig:bandv}
\end{figure}

Of particular importance are the two sequences of join- and meet-irreducible varieties $\AA_m$ and $\BB_m$, 
$m\geq 2$, as well as their duals $\ol{\AA_m}$ and $\ol{\BB_m}$ -- any proper (i.e.\ finitely generated) band 
variety not contained in $\mathcal{SL}\vee\mathcal{LZ}\vee\mathcal{RZ}$ (the `bottom cube') is a join of a pair 
of these. In fact, we shall first prove that any finite band generating one of the varieties $\AA_m$, $\BB_m$
is finitely related; by left-right duality, this will extend to $\ol{\AA_m}$ and $\ol{\BB_m}$, and then we shall 
show how to use this fact to prove finite relatedness for any finite band generating a proper subvariety of $\BB$.
The varieties $\mathcal{LRB}$ and $\mathcal{RRB}$ are the varieties of \emph{left regular} and \emph{right regular 
bands} defined by identities $xyx\id xy$ and $xyx\id yx$, respectively. Their join is the variety of \emph{regular
bands} that is the subject of Theorem 6.2 of \cite{Ma}. As we remarked earlier, our approach will be inductive 
with respect to the chains formed by varieties $\AA_m$, $\BB_m$ and their duals, with the latter result from 
\cite{Ma} serving as a basis of that induction.

For technical convenience in our main argument it will be useful to select and fix finite bands $A_m,B_m$, 
$m\geq 2$ (as well as their dual semigroups $\ol{A_m}$ and $\ol{B_m}$), such that $A_m$ generates the variety
$\AA_m$ and $B_m$ generates $\BB_m$ (For example, one can choose $A_2$ to be the two-element left zero band,
$B_2$ to be $A_2$ with an identity element adjoined, etc. Note that we can choose these bands so that 
$|A_2|<|B_2|<|A_3|< |B_3|<\dots$ holds.)  

Our main task at this moment is to describe the equational theories of the considered band varieties, following
the approach laid out in \cite{GP}. To this end, we introduce word functions $h_m$ and $i_m$ for $m\geq 2$ and
their duals, $\ol{h}_m,\ol{i}_m$, in the sense that if $\ol{\ww}$ denotes the reverse of the word $\ww$ we have 
$\ol{t}_m(\ww)=\ol{t_m(\ol{\ww})}$ for $t\in\{h,i\}$:
\begin{itemize}
\item $t_m(\es)=\es$ for all $m\geq 2$ and $t\in\{h,i\}$;
\item $h_2(\ww)$ is the first letter of $\ww$ from the left (the \emph{head} of $\ww$);
\item $i_2(\ww)$ is the word obtained from $\ww$ by retaining only the first occurrence from the left of each 
letter (the \emph{initial part} of $\ww$, also defined recursively by $i_2(\ww)=i_2s(\ww)\sigma(\ww)$);
\item for $m\geq 3$ and $t\in\{h,i\}$ we set
$$
t_m(\ww) = t_ms(\ww)\sigma(\ww)\ol{t}_{m-1}(\ww).
$$
\end{itemize}

The key feature of these functions is expressed in the following statement.

\begin{thm}[\cite{GP}]\label{xm-wp}
Let $m\geq 2$ and let $\uu,\vv$ be two words.
\begin{itemize}
\item $\AA_m$ satisfies $\uu\id\vv$ if and only if $h_m(\uu)=h_m(\vv)$. 
\item $\BB_m$ satisfies $\uu\id\vv$ if and only if $i_m(\uu)=i_m(\vv)$.
\end{itemize}
Analogous equivalences hold for the dual varieties $\ol{\AA_m}$ and $\ol{\BB_m}$ and functions $\ol{h}_m$ and
$\ol{i}_m$, respectively.
\end{thm}

The next few properties will be used in our proofs.

\begin{lem}\label{stm}
\begin{itemize}
\item[(1)] Let $t\in\{h,i\}$. If $m\geq 3$ or $t_m=i_2$ then $st_m(\ww)=t_ms(\ww)$ and $e\ol{t}_m(\ww)=\ol{t}_me(\ww)$ 
for any word $\ww$.
\item[(2)] Let $t\in\{h,i\}$. If $m\geq 4$ or $t_m=i_3$ then
$$
bt_m(\ww) = b(t_ms(\ww)\sigma(\ww)\eps(\ww)\ol{t}_{m-1}e(\ww))
$$
and 
$$
b\ol{t}_m(\ww) = b(t_{m-1}s(\ww)\sigma(\ww)\eps(\ww)\ol{t}_me(\ww))
$$
for any word $\ww$.
\item[(3)] Let $t\in\{h,i,\ol{h},\ol{i}\}$ and let $\uu,\vv$ be any words. If $m\geq 2$ then $bt_m(\uu)=bt_m(\vv)$
implies $t_m(\uu)=t_m(\vv)$.
\item[(4)] Let $t\in\{h,i\}$. If $m\geq 3$ then
$$
\ol{t}_m(\ww) = t_{m-1}(\ww)\eps(\ww)\ol{t}_me(\ww)
$$
for any word $\ww$.
\end{itemize}
\end{lem}

\begin{proof}
The first part of (1) is just \cite[Lemma 3.2(ii)]{GP}, while the second part is dual to the first one and follows
from it, as $e\ol{t}_m(\ww)=e\ol{t_m(\ol{\ww})}=\ol{st_m(\ol{\ww})}=\ol{t_ms(\ol{\ww})}=\ol{t_m(\ol{e(\ww)})}=
\ol{t}_me(\ww)$. (2)--(4) are Lemma 3.3(iii),(v), Lemma 4.1 and Lemma 3.3(ii) of \cite{GP}, respectively, 
where (2) and (3) are reformulated in terms of the operator $b$, bearing in mind Theorem \ref{thm:bwp}.
\end{proof}

\begin{lem}\label{xm-s}
Let $\XX\in\{\AA,\BB\}$. Assume that the variety $\XX_m$ satisfies the identity $\uu\id\vv$, where 
$|c(\uu)|,|c(\vv)|\geq 2$. Then it also satisfies $s(\uu)\id s(\vv)$.
\end{lem}

\begin{proof}
Consider first the case when $\XX_m=\AA_2$. Since both $\uu$ and $\vv$ contain at least two letters, we have
$h_2s(\uu)=h_2(\uu)=h_2(\vv)=h_2s(\vv)$, and the lemma follows. Otherwise, if either $\XX_m=\BB_2$ or $m\geq 3$,
let $t_m$ be the corresponding word function in the sense of Theorem \ref{xm-wp}. By the given conditions we
have $t_m(\uu)=t_m(\vv)$, which in turn implies $st_m(\uu)=st_m(\vv)$.
Then Lemma \ref{stm}(1) tells us that $t_ms(\uu)=t_ms(\vv)$ (since we assumed that $t_m\neq h_2$). 
Another application of Theorem \ref{xm-wp} yields that $\XX_m$ satisfies $s(\uu)\id s(\vv)$. 
\end{proof}

\subsection{Some auxiliary results}
\label{subsec:aux}

As Theorem \ref{thm:finrel} suggests, we will be concerned with essential operations $f:S^n\to S$ on a finite 
(idempotent)  semigroup $S$ such that all of its minors $f_{ij}$ are induced by words. In the next lemma we are
looking at some consequences of such a setting.

\begin{lem}\label{lem:ess}
Let $S$ be a finite semigroup, and let $f:S^n\to S$ be an operation depending on all of its variables such that
for any $i,j\in\nn$, $i<j$, there is a word $\ww_{ij}$ satisfying $f_{ij}=\ww_{ij}^S$.
\begin{itemize}
\item[(1)] $\{\ww_{ij}:\ 1\leq i<j\leq n\}$ is an $n$-scheme of words for the semigroup variety generated by $S$.
\item[(2)] If the variety generated by $S$ contains $\mathcal{SL}$ and $n\geq |S|+4$ 
then $c(\ww_{ij})=X_n\setminus\{x_i\}$.
\end{itemize}
\end{lem}

\begin{proof}
(1) Easy, and largely analogous to Lemma \ref{lem:scheme}, dealing with minors of operations instead of terms.

(2) Since $f$ depends on all of its variables (and thus in particular on $x_k$, $k\neq i$), there exist 
$a_1,\dots,a_{k-1},a_{k+1},\dots,a_n,b,c\in S$ such that 
$$f(a_1,\dots,a_{k-1},b,a_{k+1},\dots,a_n)\neq f(a_1,\dots,a_{k-1},c,a_{k+1},\dots,a_n).$$
The pigeonhole principle ensures that there exist $p,q\in\nn\setminus\{i,j,k\}$, $p<q$, such that $a_p=a_q$. Hence,
just as in \cite[Lemma 2.6]{Ma}, $f_{pq}=\ww_{pq}^S$ depends on $x_k$. However, by (1), $S$ satisfies the identity
$\ww_{pq}^{(ij)}\id\ww_{ij}^{(pq)}$ and so $x_k\in c(\ww_{pq}^{(ij)})=c(\ww_{ij}^{(pq)})$, ensuring that $\ww_{ij}$
must contain $x_k$, since $k\neq i$.
\end{proof} 

In the following, $n$-schemes of words $\{\ww_{ij}:\ 1\leq i<j\leq n\}$ over $X_n$ such that 
$c(\ww_{ij})=X_n\setminus\{x_i\}$ will be called \emph{essential}; so, the above result states that essential
operations on finite semigroups whose varieties contain nontrivial semilattices with all minors induced by words
give rise to essential $n$-schemes, provided $n$ is large enough.

\begin{lem}\label{lem:perm}
Let $\{\ww_{ij}:\ 1\leq i<j\leq n\}$ be an essential $n$-scheme of words for a semigroup variety $\V$ containing
$\BB_2$ (the variety of left regular bands), where $n\geq 5$. Then there exists a unique permutation $\pi$ of $\nn$
such that for any $i<j$,
$$
i_2(\ww_{ij}) = x_{\alpha_1}\cdots x_{\alpha_{n-1}},
$$
where the sequence $\pi^{(ij)}=(\alpha_1,\dots,\alpha_{n-1})$ is obtained from $(1\pi,\dots,n\pi)$ by replacing $i$ 
by $j$ and then deleting the right one of the two occurrences of $j$. Furthermore, if the given essential scheme 
comes from a word $\ww$ then we must have $i_2(\ww)=x_{1\pi}\cdots x_{n\pi}$.
\end{lem}

\begin{proof}
By the very definition of an (essential) $n$-scheme of words for a variety, any scheme for a variety $\V$ is also
a scheme for any of its subvarieties, and so is the case for the given scheme with respect to $\BB_2$. As already
discussed, $\BB_2$ is generated by the 3-element band $B_2$ (obtained by adjoining an identity element to the
2-element left zero band), and by Lemma 6.1 of \cite{Ma}, $B_2$ is finitely related with degree at most 4. By
\cite[Lemma 2.6]{DJPSz}, there exists a unique operation $f:B_2^n\to B_2$ such that $f_{ij}=\ww_{ij}^{B_2}$ for all
$1\leq i<j\leq n$. To see that $f$ depends on all of its variables, fix $k\in\nn$ and $p,q\in\nn\setminus\{k\}$,
$p<q$. As $f_{pq}=\ww_{pq}^{B_2}$ and $x_k\in c(\ww_{pq})$, $f_{pq}$ depends on $x_k$, so there exist $a_1,\dots,
a_{k-1},a_{k+1},\dots,a_n,b,c\in B_2$ such that $f_{pq}(a_1,\dots,a_{k-1},b,a_{k+1},\dots,a_n)\neq f_{pq}(a_1,\dots,
a_{k-1},c,a_{k+1},\dots,a_n)$. This immediately implies that $f$ depends on $x_k$.

Hence, by \cite[Lemma 6.1]{Ma}, $f$ is induced by a word $\ww$ such that $c(\ww)=X_n$. Now it is routine to
see that the permutation $\pi\in\mathbb{S}_\nn$ satisfying 
$$
i_2(\ww) = x_{1\pi}\cdots x_{n\pi}
$$
meets all the requirements of the lemma. It is unique because if $\pi'$ would be another permutation with the
required properties, then it would follow that the given scheme (with respect to $\BB_2$) also comes from the 
word $\ww'=x_{1\pi'}\cdots x_{n\pi'}$, whence the uniqueness of $f$ implies that $(\ww')^{B_2}=f$. Therefore,
$\BB_2$ satisfies $\ww'\id\ww$, implying $i_2(\ww')=i_2(\ww)$ and so $\pi'=\pi$.
\end{proof}

For an essential scheme of words satisfying the assumptions of the above lemma, we will call $\pi$ the \emph{associated
permutation} of the scheme.

We proceed with a technical lemma before we learn another important property of essential $n$-schemes of words over
irreducible band varieties.

\begin{lem}\label{lem:spq}
Let $k\in\nn$, and let $\ww\in X_n^+$ be such that $c(\ww)=X_n\setminus\{x_k\}$ and $\sigma(\ww)=x_l$. 
Furthermore, let $p,q\in\nn\setminus\{l\}$ such that $p<q$. Then $s(\ww)^{(pq)}=s(\ww^{(pq)})$.
\end{lem}

\begin{proof}
We start by writing $\ww$ in the form
$$
\ww=s(\ww)x_l\uu
$$
for some word $\uu$, implying $\ww^{(pq)}=s(\ww)^{(pq)}x_l\uu^{(pq)}$. If $q\neq k$ then 
$c(\ww^{(pq)})=X_n\setminus\{x_k,x_p\}$ and $c(s(\ww)^{(pq)})=X_n\setminus\{x_l,x_k,x_p\}$, which
immediately gives $s(\ww^{(pq)})=s(\ww)^{(pq)}$. The same conclusion follows if $q=k$ upon noticing
that in such a case we have $c(\ww^{(pq)})=X_n\setminus\{x_p\}$ and $c(s(\ww)^{(pq)})=X_n\setminus\{x_l,x_p\}$.
\end{proof}

\begin{lem}\label{lem:s'}
Let $n\geq 5$, let $\mathcal{S}=\{\ww_{ij}:\ 1\leq i<j\leq n\}$ be an essential $n$-scheme of words for 
$\XX_m$, where $\XX\in\{\AA,\BB\}$ and either $m\geq 3$ or $\XX_m=\BB_2$. Let $\pi$ be the associated permutation 
of $\mathcal{S}$ (which exists by Lemma \ref{lem:perm}) and let $l=n\pi$. Then
$$
\mathcal{S}' = \{s(\ww_{ij}):\ 1\leq i<j\leq n,\ l\not\in\{i,j\}\}
$$
is an $(n-1)$-scheme for $\XX_m$ over variables $X_n\setminus\{x_l\}$.
\end{lem}

\begin{proof}
Note that if $l\not\in\{i,j\}$ then $\sigma(\ww_{ij})=x_l$ and so $c(s(\ww_{ij}))=X_n\setminus\{x_i,x_l\}$.
Thus (D) holds.

Furthermore, if $p<q$ are such that $\{p,q\}\cap\{i,j,l\}=\es$ then by Lemma \ref{lem:spq} we have 
$s(\ww_{ij})^{(pq)}=s(\ww_{ij}^{(pq)})$ and $s(\ww_{pq})^{(ij)}=s(\ww_{pq}^{(ij)})$. Further the assumption that
$\XX_m$ satisfies $\ww_{ij}^{(pq)}\id\ww_{pq}^{(ij)}$ implies, by Lemma \ref{xm-s}, the identity 
$s(\ww_{ij}^{(pq)})\id s(\ww_{pq}^{(ij)})$. Hence, $\XX_m$ satisfies $s(\ww_{ij})^{(pq)}\id s(\ww_{pq})^{(ij)}$,
and so $\mathcal{S}'$ satisfies (C1). 

Now let $i<j<p$ be such that $l\not\in\{i,j,p\}$. Again, Lemma \ref{lem:spq} implies that 
$s(\ww_{ij})^{(jp)}=s(\ww_{ij}^{(jp)})$, $s(\ww_{jp})^{(ip)}=s(\ww_{jp}^{(ip)})$ and 
$s(\ww_{ip})^{(jp)}=s(\ww_{ip}^{(jp)})$ (the lemma was applied over the alphabet $X_n\setminus\{x_i\}$ 
in the first and the third case, while it was applied over $X_n\setminus\{x_j\}$ in the second case). Since, 
by assumption, $\XX_m$ satisfies the identities $\ww_{ij}^{(jp)}\id\ww_{jp}^{(ip)}\id\ww_{ip}^{(jp)}$, it also 
satisfies the identities $s(\ww_{ij}^{(jp)})\id s(\ww_{jp}^{(ip)})\id s(\ww_{ip}^{(jp)})$ by Lemma \ref{xm-s}. 
We conclude that $\XX_m$ satisfies $s(\ww_{ij})^{(jp)}\id s(\ww_{jp})^{(ip)}\id s(\ww_{ip})^{(jp)}$, hence (C2) 
holds for $\mathcal{S}'$ as well.
\end{proof}

We are now ready and equipped to present our main arguments.

\section{The main proofs}
\label{sec:main}

\begin{thm}\label{thm:Tn}
For any $m\geq 2$ and $T\in\{A,B\}$, the bands $T_m$ and $\ol{T_m}$ are finitely related.
\end{thm}

\begin{proof}
We use induction on $m$. If $m=2$, the result already follows from \cite[Theorem 6.2]{Ma}. Hence fix $m\geq 3$
and assume that the statement is true for values of indices up to $m-1$.

We are going to show that $T_m$ is finitely related by verifying condition (3') in Theorem \ref{thm:finrel}
with $n\geq n_0 = \max(|T_m|+4,m+3)$. So, under these assumptions, let $f:T_m^n\to T_m$ be an operation that
depends on all of its variables such that for all $i,j\in\nn$, $i<j$, we have that the operation $f_{ij}$ is 
induced by a word. Let us select words $\ww_{ij}\in X_n^+$, $1\leq i<j\leq n$, such that $f_{ij}=\ww_{ij}^{T_m}$. 
Since the variety $\XX_m$ generated by $T_m$ contains $\mathcal{SL}$, Lemma \ref{lem:ess} tells us that
$c(\ww_{ij})=X_n\setminus\{x_i\}$, so that $\mathcal{S}=\{\ww_{ij}:\ 1\leq i<j\leq n\}$ is an essential $n$-scheme 
of words for $\XX_m$.

Since $\BB_2$ is contained in $\XX_m$ and $n\geq m+3>5$, Lemma \ref{lem:perm} ensures the existence (and uniqueness)
of the associated permutation $\pi$ of the scheme $\mathcal{S}$. Let us denote $k=(n-1)\pi$ and $l=n\pi$.

Now, since $\mathcal{S}$ is an $n$-scheme of words for $\XX_m$, it is also an $n$-scheme of words for any of its
subvarieties, and thus in particular for $\ol{\XX_{m-1}}$. This variety is generated by the band $\ol{T_{m-1}}$
which is finitely related by the inductive assumption. Hence, Theorem \ref{thm:finrel} implies that $\mathcal{S}$,
considered as a scheme for $\ol{\XX_{m-1}}$, comes from a word, say $\uh$. Equivalently, there is a word operation
$g:\ol{T_{m-1}}^n\to\ol{T_{m-1}}$ such that $g_{ij}=\ww_{ij}^{\ol{T_{m-1}}}$ holds for all $1\leq i<j\leq n$.

Similarly, by Lemma \ref{lem:s'}, the family of words $\mathcal{S}'=\{s(\ww_{ij}):\ 1\leq i<j\leq n,\ l\not\in
\{i,j\}\}$ is an $(n-1)$-scheme of words for $\XX_m$ (and thus for $\ol{\XX_{m-1}}$) over the set of variables
$X_n\setminus\{x_l\}$. Again, by employing the inductive hypothesis (and bearing in mind that $n-1\geq (m-1)+3$
and $n-1\geq (|T_m|-1)+4\geq |\ol{T_{m-1}}|+4$), we find a word $\ut$ such that $\mathcal{S}'$ comes from $\ut$;
in other words, there is a word operation $h:\ol{T_{m-1}}^{n-1}\to\ol{T_{m-1}}$ with $h_{ij}=
s(\ww_{ij})^{\ol{T_{m-1}}}$ for all $i<j$ such that $l\not\in\{i,j\}$.

Now define
$$
\ww = s(\ww_{k'l'})x_k\ut x_l\uh,
$$
where $k'=\min(k,l)$ and $l'=\max(k,l)$. We are going to argue that $f=\ww^{T_m}$ whence Theorem \ref{thm:finrel} 
yields that $T_m$ is finitely related.

\begin{cla}
If $p,q\in\nn$, $p<q$, are such that $\{p,q\}\cap\{k,l\}=\es$ then the identity 
$$ \ww^{(pq)} \id \ww_{pq} $$
holds in $\XX_m$ (and so in $T_m$).
\end{cla}

\begin{proof}[Proof of Claim.]
We have 
\begin{equation}\label{eq1}
\ww^{(pq)} = (s(\ww_{k'l'}))^{(pq)}x_k\ut^{(pq)}x_l\uh^{(pq)}=s(\ww_{k'l'}^{(pq)})x_k\ut^{(pq)}x_l\uh^{(pq)},
\end{equation}
where the second equality follows by Lemma \ref{lem:spq} using $c(\ww_{k'l'})=X_n\setminus\{x_{k'}\}$. Furthermore, 
by (C1), $\XX_m$ satisfies $\ww_{k'l'}^{(pq)}\id \ww_{pq}^{(k'l')}$, so by Lemma \ref{xm-s} it also satisfies 
$s(\ww_{k'l'}^{(pq)})\id s(\ww_{pq}^{(k'l')})$. Since $\{p,q\}\cap\{k,l\}=\es$, the sequence $\pi^{(pq)}$
must have the form $(\alpha_1,\dots,\alpha_{n-3},k,l)$ (where the subsequence $(\alpha_1,\dots,\alpha_{n-3})$
is a certain permutation of $\nn\setminus\{k,l,p\}$). Therefore, we may highlight the first occurrence from
the left of each letter from $c(\ww_{pq})$ in $\ww_{pq}$ by writing
$$
\ww_{pq}=x_{\alpha_1}\uu_1x_{\alpha_2}\cdots x_{\alpha_{n-3}}\uu_{n-3}x_k\uu_{n-2}x_l\vv
$$
for some (possibly empty) words $\uu_1,\dots,\uu_{n-2},\vv$. This yields
$$
\ww_{pq}^{(k'l')}=x_{\alpha_1}\uu_1x_{\alpha_2}\cdots x_{\alpha_{n-3}}\uu_{n-3}
x_{l'}\uu_{n-2}^{(k'l')}x_{l'}\vv^{(k'l')},
$$
implying 
$$
s(\ww_{pq}^{(k'l')})=x_{\alpha_1}\uu_1x_{\alpha_2}\cdots x_{\alpha_{n-3}}\uu_{n-3}=s^2(\ww_{pq}).
$$
Bearing in mind \eqref{eq1}, we deduce that
$$
\ww^{(pq)} \id  s(\ww_{pq}^{(k'l')})x_k\ut^{(pq)}x_l\uh^{(pq)} = s^2(\ww_{pq})x_k\ut^{(pq)}x_l\uh^{(pq)}
$$
holds in $\XX_m$. By applying Theorem \ref{xm-wp} and the recursion for $t_m$ twice, we get
\begin{align*}
t_m(\ww^{(pq)}) &= t_m\bigl(s^2(\ww_{pq})x_k\ut^{(pq)}x_l\uh^{(pq)}\bigr) \\
&= t_m\bigl(s^2(\ww_{pq})x_k\ut^{(pq)}\bigr)x_l\ol{t}_{m-1}\bigl(s^2(\ww_{pq})x_k\ut^{(pq)}x_l\uh^{(pq)}\bigr)\\
&= t_ms^2(\ww_{pq})x_k\ol{t}_{m-1}\bigl(s^2(\ww_{pq})x_k\ut^{(pq)}\bigr)x_l\ol{t}_{m-1}\bigl(s^2(\ww_{pq})x_k\ut^{(pq)}x_l\uh^{(pq)}\bigr).
\end{align*}
However, by induction assumption on $\mathcal{S}'$ for $\ol{\XX_{m-1}}$ we have that $\ol{\XX_{m-1}}$ satisfies 
$\ut^{(pq)}\id s(\ww_{pq})$, and by induction assumption on $\mathcal{S}$ for $\ol{\XX_{m-1}}$ we have that 
$\ol{\XX_{m-1}}$ satisfies $\uh^{(pq)}\id \ww_{pq}$. We know that $\sigma(\ww_{pq})=x_l$, $s(\ww_{pq})= 
x_{\alpha_1}\cdots x_{\alpha_{n-3}}\uu_{n-3}x_k\uu_{n-2}$, $\sigma s(\ww_{pq})=x_k$, and $s^2(\ww_{pq})=
x_{\alpha_1}\cdots x_{\alpha_{n-3}}\uu_{n-3} $. Thus
$$
s^2(\ww_{pq})x_k\ut^{(pq)} \id s^2(\ww_{pq})\sigma s(\ww_{pq})s(\ww_{pq}) \id s(\ww_{pq})
$$
and consequently
$$
s^2(\ww_{pq})x_k\ut^{(pq)}x_l\uh^{(pq)} \id s(\ww_{pq})\sigma(\ww_{pq})\ww_{pq} \id \ww_{pq}
$$
holds in $\ol{\XX_{m-1}}$ (here we used the idempotent law, upon noticing that $s^2(\ww_{pq})\sigma s(\ww_{pq})$
is a prefix of $s(\ww_{pq})$, while $s(\ww_{pq})\sigma(\ww_{pq})$ is a prefix of $\ww_{pq}$). Hence,
\begin{align*}
t_m(\ww^{(pq)}) &= t_ms^2(\ww_{pq})x_k\ol{t}_{m-1}s(\ww_{pq})x_l\ol{t}_{m-1}(\ww_{pq}) \\
&= t_ms(\ww_{pq})x_l\ol{t}_{m-1}(\ww_{pq}) = t_m(\ww_{pq}).
\end{align*}
Another application of Theorem \ref{xm-wp} concludes the proof of the claim that $\ww^{(pq)}\id\ww_{pq}$ holds
in $\XX_m$.
\phantom\qedhere\end{proof}

Let now $\ol{a}=(a_1,\dots,a_n)\in T_m^n$ be arbitrary. As $n\geq |T_m|+4>|T_m|+3$, by the pigeohole principle there
are $p,q\in\nn$, $p<q$, such that $a_p=a_q$ and $\{p,q\}\cap\{k,l\}=\es$. Thus
$$
f(\ol{a}) = f_{pq}(\ol{a}) = \ww_{pq}^{T_m}(\ol{a}) = (\ww^{(pq)})^{T_m}(\ol{a}) = (\ww^{T_m})_{pq}(\ol{a}) = 
\ww^{T_m}(\ol{a}),
$$
where the previous claim was used in the third equality above. Therefore, $f=\ww^{T_m}$, showing that $T_m$
is finitely related.

The result for $\ol{T_m}$ follows by left-right duality or simply by the observation that if a finite semigroup $S$
is finitely related, so is its dual semigroup $\ol{S}$. (Indeed, if $\{\uu_{ij}:\ 1\leq i<j\leq n\}$ is an $n$-scheme
of words for the variety $\ol{\V}$ generated by $\ol{S}$ then it is a routine to show that 
$\{\ol{\uu_{ij}}:\ 1\leq i<j\leq n\}$ is an $n$-scheme of words for the variety $\V$ generated by $S$; so, if the
latter scheme comes from a word $\uu$ it follows immediately that the former one comes from $\ol{\uu}$. Now an
appeal to Theorem \ref{thm:finrel} gives the required result.)
\end{proof}

\begin{thm}\label{thm:prod}
For any $m\geq 3$, the bands $A_m\times\ol{A_m}$, $A_m\times\ol{B_{m-1}}$, $B_{m-1}\times\ol{A_m}$, 
$B_m\times\ol{B_m}$, $B_m\times\ol{A_m}$ and $A_m\times\ol{B_m}$ are finitely related.
\end{thm}

\begin{rmk}
Note that this theorem does not follow automatically from the previous one, as it was shown in \cite[Example 6.3]{DJPSz}
that the finitely related property is in general not preserved by direct products.
\end{rmk}

\begin{proof}
Generally, we are going to argue that a finite band of the form $T_m\times\ol{U_r}$ is finitely related,
where $T,U\in\{A,B\}$ and $r\in\{m-1,m\}$, with combinations as in the formulation; since the third case is
dual to the second and sixth case to the fifth, these two may be safely omitted by our previous left-right
duality remarks. Note that with these restrictions if $T_m$ generates $\XX_m$ and $\ol{U_r}$ generates $\ol{\YY_r}$
($\XX,\YY\in\{\AA,\BB\}$) then $\ol{\YY_r}$ is always a subvariety of $\ol{\XX_m}$ containing $\ol{\XX_{m-1}}$.
Equivalently, $\XX_m$ is a subvariety of $\YY_{r+1}$ containing $\YY_r$.
Also, $\BB_2$ can be assumed to be a subvariety of both $\XX_m$ and $\YY_r$.

So, let $n_0=\max(|T_m||U_r|+4,m+3)$ and assume that for some $n\geq n_0$, $f$ is an 
$n$-ary term operation of  $T_m\times\ol{U_r}$ depending on all of its variables, with the property that 
for all $1\leq i<j\leq n$, $f_{ij}$ is induced by a word $\ww_{ij}\in X_n^+$. 

\begin{cla}
There exist operations $g:T_m^n\to T_m$ and $h:\ol{U_r}^n\to\ol{U_r}$ such that
$$
f((a_1,b_1),\dots,(a_n,b_n)) = \left(g(a_1,\dots,a_n),h(b_1,\dots,b_n)\right)
$$
for all $a_1,\dots,a_n\in T_m$ and $b_1,\dots,b_n\in\ol{U_r}$, and, furthermore, for all $1\leq i<j\leq n$, 
the operations $g_{ij}$ and $h_{ij}$ are induced by $\ww_{ij}$ on $T_m$ and $\ol{U_r}$, respectively.
\end{cla}

\begin{proof}[Proof of Claim.]
By Lemma \ref{lem:ess}(1), $\mathcal{S}=\{\ww_{ij}:\ 1\leq i<j\leq n\}$ is an $n$-scheme of words for the variety 
generated by $T_m\times\ol{U_r}$, and thus for each of the varieties generated individually by $T_m$ and $\ol{U_r}$. 
Now by \cite[Lemma 2.6]{DJPSz} there are (unique) $n$-ary operations $g$ and $h$ on $T_m$  and $\ol{U_r}$, 
respectively, such that $g_{ij}=\ww_{ij}^{T_m}$ and $h_{ij}=\ww_{ij}^{\ol{U_r}}$ for all $1\leq i<j\leq n$.

As $n>|T_m||U_r|$, for arbitrary $(a_1,b_1),\dots,(a_n,b_n)\in T_m\times\ol{U_r}$ there are $p,q\in\nn$, $p<q$, 
such that $(a_p,b_p)=(a_q,b_q)$. Hence, 
\begin{align*}
f((a_1,b_1),\dots,(a_n,b_n)) &= f_{pq}((a_1,b_1),\dots,(a_n,b_n)) \\
&= \ww_{pq}^{T_m\times\ol{U_r}}((a_1,b_1),\dots,(a_n,b_n)) \\
&= \left(\ww_{pq}^{T_m}(a_1,\dots,a_n),\ww_{pq}^{\ol{U_r}}(b_1,\dots,b_n)\right) \\
&= (g_{pq}(a_1,\dots,a_n), h_{pq}(b_1,\dots,b_n)) \\
&= (g(a_1,\dots,a_n), h(b_1,\dots,b_n)).
\end{align*}
By the very construction of $g$ and $h$, the claim follows.
\phantom\qedhere\end{proof}

By the choice of $n_0$ and Theorem \ref{thm:Tn} (in fact, by $n$ satisfying the assumptions in its proof),
both $g$ and $h$ are induced by words, say $\uu$ and $\vv$. We claim that $f$ is induced on $T_m\times\ol{U_r}$
by the word $\ww=\uu\vv$. We are going to prove that for all $1\leq i<j\leq n$, both $T_m$ and $\ol{U_r}$
satisfy $\ww^{(ij)}\id\ww_{ij}$. Similarly as in the final part of the proof of the previous theorem, this
will suffice to establish the theorem, because then for arbitrary $(a_1,b_1),\dots,(a_n,b_n)\in T_m\times\ol{U_r}$
the pigeonhole principle provides $p<q$ with $(a_p,b_p)=(a_q,b_q)$, so that for $\ol{c}=((a_1,b_1),\dots,(a_n,b_n))$,
$\ol{a}=(a_1,\dots,a_n)$ and $\ol{b}=(b_1,\dots,b_n)$ we have
\begin{align*}
f(\ol{c}) &= f_{pq}(\ol{c}) = (g_{pq}(\ol{a}),h_{pq}(\ol{b}))=(\ww_{pq}^{T_m}(\ol{a}),\ww_{pq}^{\ol{U_r}}(\ol{b}))\\
&= ((\ww^{(pq)})^{T_m}(\ol{a}),(\ww^{(pq)})^{\ol{U_r}}(\ol{b}))\\
&=(\ww^{(pq)})^{T_m\times\ol{U_r}}(\ol{c}) = (\ww^{T_m\times\ol{U_r}})_{pq}(\ol{c})=\ww^{T_m\times\ol{U_r}}(\ol{c}),
\end{align*}
showing that $f=\ww^{T_m\times\ol{U_r}}$. 

Let us start by noting that the scheme $\mathcal{S}$ is essential (by Lemma \ref{lem:ess}(2)), which immediately 
implies $c(\uu)=c(\vv)=X_n$, so $s(\ww)=s(\uu)$, $\sigma(\ww)=\sigma(\uu)$, $\eps(\ww)=\eps(\vv)$ and $e(\ww)=e(\vv)$.
This yields $s(\ww^{(ij)})=s(\uu^{(ij)})$ and $e(\ww^{(ij)})=e(\vv^{(ij)})$.
Furthermore, $i_2(\ww)=i_2(\uu)$, implying $i_2(\ww^{(ij)})=i_2(\uu^{(ij)})=i_2(\ww_{ij})$ by Lemma \ref{lem:perm}
because $\XX_m$ contains $\BB_2$. Thus $\sigma(\ww^{(ij)})=\sigma(\uu^{(ij)})=\sigma(\ww_{ij})$ and, dually, 
$\eps(\ww^{(ij)})=\eps(\vv^{(ij)})=\eps(\ww_{ij})$. Also, if $t_m$ is the word function corresponding to the variety 
$\XX_m$ (in the sense of Theorem \ref{xm-wp}), then we know, by construction of words $\uu$ and $\vv$, that $\XX_m$ 
satisfies $\uu^{(ij)}\id\ww_{ij}$ --- so that $t_m(\uu^{(ij)})=t_m(\ww_{ij})$ --- while $\ol{\YY_r}$ satisfies 
$\vv^{(ij)}\id\ww_{ij}$. By assumptions made earlier in this proof, the latter identity holds in $\ol{\XX_{m-1}}$, 
implying $\ol{t}_{m-1}(\vv^{(ij)})=\ol{t}_{m-1}(\ww_{ij})$. 

Note that we have $\ol{h}_2(\ww^{(ij)})=\ol{h}_2(\vv^{(ij)})$ since $e(\ww)=e(\vv)$.
So, if $\XX_m=\AA_3$ we deduce, by Lemma \ref{stm}(1),
\begin{align*}
h_3(\ww^{(ij)}) & = h_3s(\uu^{(ij)})\sigma(\uu^{(ij)})\ol{h}_2(\ww^{(ij)})\\
& = sh_3(\uu^{(ij)})\sigma(\uu^{(ij)})\ol{h}_2(\vv^{(ij)})\\
& = sh_3(\ww_{ij})\sigma(\ww_{ij})\ol{h}_2(\ww_{ij})\\
& = h_3s(\ww_{ij})\sigma(\ww_{ij})\ol{h}_2(\ww_{ij})=h_3(\ww_{ij}).
\end{align*}
On the other hand, if $\XX_m\in\{\AA_k:\ k\geq 4\}\cup\{\BB_k:\ k\geq 3\}$ then by items (1) and (2) of
Lemma \ref{stm} we obtain
\begin{align*}
bt_m(\ww^{(ij)}) &= b\left[t_ms(\uu^{(ij)})\sigma(\uu^{(ij)})\eps(\vv^{(ij)})\ol{t}_{m-1}e(\vv^{(ij)})\right]\\
& = b\left[st_m(\uu^{(ij)})\sigma(\uu^{(ij)})\eps(\vv^{(ij)})e\ol{t}_{m-1}(\vv^{(ij)})\right]\\
& = b\left[st_m(\ww_{ij})\sigma(\ww_{ij})\eps(\ww_{ij})e\ol{t}_{m-1}(\ww_{ij})\right]\\
& = b\left[t_ms(\ww_{ij})\sigma(\ww_{ij})\eps(\ww_{ij})\ol{t}_{m-1}e(\ww_{ij})\right] = bt_m(\ww_{ij}).
\end{align*}
But then Lemma \ref{stm}(3) implies that $t_m(\ww^{(ij)})=t_m(\ww_{ij})$, i.e.\ that $\XX_m$ (and thus $T_m$) 
satisfies $\ww^{(ij)}\id\ww_{ij}$. 

The proof that $\ol{\YY_r}$ satisfies $\ww^{(ij)}\id\ww_{ij}$ is similar, albeit with slight differences. Let now $t_r$
be the word function corresponding to $\YY_r$ in the sense of Theorem \ref{xm-wp}. The fact that $\XX_m$ (and so $\YY_r$) 
satisfies $\uu^{(ij)}\id\ww_{ij}$ implies $t_r(\uu^{(ij)})=t_r(\ww_{ij})$. Furthermore, \cite[Lemma 3.5]{GP} tells us
that $t_{r-1}(\uu^{(ij)})=t_{r-1}(\ww_{ij})$ holds as well, provided $r\geq 3$. In turn, $\ol{\YY_r}$ satisfies 
$\vv^{(ij)}\id\ww_{ij}$, thus $\ol{t}_r(\vv^{(ij)})=\ol{t}_r(\ww_{ij})$.

Now we have three subcases to consider. First, let $\ol{\YY_r}=\ol{\BB_2}$ (which can happen, just as the next subcase, 
only if $\XX_m=\AA_3$). Since $c(\uu)=c(\vv)=X_n$, we have $\ol{i}_2(\ww^{(ij)})=\ol{i}_2(\uu^{(ij)}\vv^{(ij)})=
\ol{i}_2(\vv^{(ij)})$. On the other hand, we have already concluded that $\ol{\YY_r}$ satisfies 
$\vv^{(ij)}\id\ww_{ij}$, so $\ol{i}_2(\vv^{(ij)})=\ol{i}_2(\ww_{ij})$ yielding $\ol{i}_2(\ww^{(ij)})=
\ol{i}_2(\ww_{ij})$, as required. Our second subcase is $\ol{\YY_r}=\ol{\AA_3}$, when $h_2(\ww^{(ij)})=
h_2(\uu^{(ij)})$ because of $s(\ww)=s(\uu)$. By items (1) and (4) of Lemma \ref{stm} we have:
\begin{align*}
\ol{h}_3(\ww^{(ij)}) & = h_2(\ww^{(ij)})\eps(\vv^{(ij)})\ol{h}_3e(\vv^{(ij)})\\
& = h_2(\uu^{(ij)})\eps(\vv^{(ij)})e\ol{h}_3(\vv^{(ij)})\\
& = h_2(\ww_{ij})\eps(\ww_{ij})e\ol{h}_3(\ww_{ij})\\
& = h_2(\ww_{ij})\eps(\ww_{ij})\ol{h}_3e(\ww_{ij})=\ol{h}_3(\ww_{ij}).
\end{align*}
Finally, let $\ol{\YY_r}\in\{\ol{\AA_k}:\ k\geq 4\}\cup\{\ol{\BB_k}:\ k\geq 3\}$. Then, by invoking parts (1) and (2)
of Lemma \ref{stm} once again, we deduce
\begin{align*}
b\ol{t}_r(\ww^{(ij)}) &= b\left[t_{r-1}s(\uu^{(ij)})\sigma(\uu^{(ij)})\eps(\vv^{(ij)})\ol{t}_re(\vv^{(ij)})\right]\\
& = b\left[st_{r-1}(\uu^{(ij)})\sigma(\uu^{(ij)})\eps(\vv^{(ij)})e\ol{t}_r(\vv^{(ij)})\right]\\
& = b\left[st_{r-1}(\ww_{ij})\sigma(\ww_{ij})\eps(\ww_{ij})e\ol{t}_r(\ww_{ij})\right]\\
& = b\left[t_{r-1}s(\ww_{ij})\sigma(\ww_{ij})\eps(\ww_{ij})\ol{t}_re(\ww_{ij})\right] = b\ol{t}_r(\ww_{ij}),
\end{align*}
whence Lemma \ref{stm}(3) implies that $\ol{t}_r(\ww^{(ij)})=\ol{t}_r(\ww_{ij})$.

Hence, both $\XX_m$ and $\ol{\YY_r}$ (with restrictions on $m,r$ as described at the beginning of the proof)
satisfy the identity $\ww^{(ij)}\id\ww_{ij}$. As remarked earlier, this completes the proof that $f$ is 
induced by $\ww$ and confirms the theorem.
\end{proof}

The proof of Theorem \ref{main} is now immediate: if $S$ is a finite band then it generates the same variety
as one of the bands covered by \cite[Theorem 6.2]{Ma}, Theorem \ref{thm:Tn} and Theorem \ref{thm:prod}. By 
\cite[Theorem 2.11]{DJPSz}, $S$ is finitely related.

\begin{ackn}
The author is grateful to a thorough and perceptive referee for a handful of comments which significantly 
improved the paper.
\end{ackn}


\end{document}